\documentclass{amsart}
\usepackage{amssymb}
\usepackage{verbatim}
\usepackage[colorlinks=true, pdfstartview=FitV, linkcolor=blue, 
            citecolor=red, urlcolor=red]{hyperref}

\newtheorem{theorem}{Theorem}[section]

\newtheorem{corollary}[theorem]{Corollary}
\newtheorem{definition}[theorem]{Definition}
\let\olddefinition\definition
\renewcommand{\definition}{\olddefinition\normalfont}
\newtheorem{lemma}[theorem]{Lemma}
\newtheorem{remark}[theorem]{Remark}
\let\oldremark\remark
\renewcommand{\remark}{\oldremark\normalfont}

\numberwithin{equation}{section}

\def\tr{\ensuremath {\textnormal{tr}}}
\def\Q{\ensuremath {{\bf Q}}}

\def\C{\ensuremath {{\bf C}}}

\def\R{\ensuremath {{\bf R}}}

\def\A{\ensuremath {{\bf A}}}
\def\a{\ensuremath {\mathfrak a}}
\def\Z{\ensuremath {{\bf Z}}}

\def\f{\ensuremath {\mathfrak f}}
\def\O{\ensuremath {\mathcal O}}
\def\p{\ensuremath {\mathfrak p}}

\newcommand\be{\begin{equation}}
\newcommand\ee{\end{equation}}

\title[Weil distributions, explicit formulas, and the continuous spectrum]{Weil distributions, explicit formulas, and the continuous spectrum of $SL(2)$}
\author{Tian An Wong}
\date{\today}
\email{wongtianan@math.ubc.ca}
\address{University of British Columbia, Vancouver, BC}


\begin{document}

\title[Explicit formulas for the spectral side of the trace formula]{
Explicit formulas for the spectral side of the trace formula of SL(2)
}




\date{\today}

\maketitle

\begin{abstract}
The continuous spectrum to the spectral side of the Arthur-Selberg trace formula is described in terms of intertwining operators, whose normalising factors involve quotients of $L$-functions. In this paper, we derive two expressions in the case of $SL(2)$ over a number field in terms of the Riemann-Weil explicit formula: as a sum over zeroes of the associated $L$-functions, and as a sum of adelic distributions on Weil groups. As an application, we obtain an expression for a lower bound for the sums over zeroes with respect to the truncation parameter for Eisenstein series.

\end{abstract}

\section{Introduction}

The conjectural interpretation of the Selberg trace formula as an explicit formula for zeroes of $L$-functions has been pursued by many authors, for example \cite{Gol,Den2,Con,Mey}. Indeed, the resemblance of the two was already pointed out in Selberg's original paper \cite{S}. On the other hand, it is also well known that the contribution of the continuous spectrum of $G$ to the spectral side of the Selberg trace formula for $G$ is given by an operator that is normalised by a quotient of $L$-functions. Different authors refer to the operator as the Eisenstein matrix, scattering matrix, or intertwining operator (e.g., \cite{Rez,Mas,JL}). As the logarithmic derivative of this operator enters into the spectral side of the trace formula by the Maa\ss-Selberg relation, there is a natural connection here to the zeroes of the given $L$-functions.

The goal of this paper is to write down the explicit formulas that are expected to arise from the latter perspective for $G=SL(2)$, over a number field. In contrast to previous works, we adopt the representation theoretic approach, which allows us to work over number fields, and the scattering matrix is more easily dealt with as the intertwining operator of the principal series representation of $SL(2)$. As an application, we shall obtain an expression for a lower bound for the sum over zeroes of Hecke $L$-functions. The techniques presented here readily generalise to the continuous spectrum of a general reductive group, for the normalised intertwining operator of the principal series representation induced from a maximal parabolic subgroup.

\subsection{Main results} 

Let $F$ be a number field with $r_1$ real and $2r_2$ complex embeddings, and denote by $\O_F$ its ring of integers, $d_F$ its discriminant, $\A=\A_F$ its ring of adeles, and $N=N_{F/Q}$ the norm map. Define the von Mangoldt function for number fields on nonzero integral ideals $\mathfrak a\subset\O_F$,
\[
\Lambda({\mathfrak a})=\begin{cases}\log N(\p) & \text{if }{\mathfrak a}=p^k\text{ for some }k\geq 1\\ 0 & \text{otherwise.}\end{cases}
\]
To fix notation, we introduce the global Hecke $L$-function associated to $\chi$, $L(s,\chi)=\prod_v L_v(s,\chi_v),$ the product taken over all places $v$ of $F$. It satisfies the functional equation
\[
L(s,\chi)=\varepsilon(s,\chi,\psi)L(1-s,\overline{\chi})\label{fe}
\]
where the epsilon factor is defined as $\varepsilon(s,\chi,\psi)=W(\chi)|N_{F/\Q}(\f(\chi))d_F|^{s-\frac12}$, where $W(\chi)$ is the root number, $\f(\chi)$ the conductor of $\chi$, and $\psi$ is a fixed additive character of $F$. We also define the completed $L$-function $\Lambda(s)=|N_{F/\Q}(\f(\chi))d_F|^{s/2}L(s,\chi)$, and denote by $\zeta_F(s,\chi)$ the product over all finite primes of $L_v(s,\chi)$.

The local $L$-factors are given as follows: If $v=\p$ a prime ideal of $F$, then
\[
L_v(s,\chi_v)=(1-\chi_v(\p)N_{F/\Q}(\p)^{-s})^{-1}
\]
if $\chi_v$ is unramified, and 1 if it is ramified. If $v$ is a real place, the field $F_v\simeq \R$ has no non-trivial automorphisms, thus a character $\chi_v$ of $F_v^\times$ can be identified with one of $\R^\times$, hence necessarily of the form  $\chi_v(t)=\text{sgn}(t)^n|t|^w,$ where $n\in\{0,1\}$, $w\in\C$, sgn$(t)$ is the usual sign of $t$, and one has
\[
L_v(s,\chi_v)=\Gamma_\R(s)=\pi^{-\frac{s}{2}}\Gamma(\frac{s+w+n}{2}).
\]
If $v$ is complex, the field $F_v\simeq\C$ has two possible identifications. Choosing one, the character $\chi_v$ will be a character of $\C^\times$, necessarily of the form 
$\chi_v(z)=\text{arg}(z)^n|z|^{2w},
$ where $n\in\Z$, $w\in\C$, arg$(z)$ is the complex argument $z/|z|_\C=z^{1/2}\bar{z}^{-1/2}$, giving
\[
L_v(s,\chi_v)=\Gamma_\C(s)=(2\pi)^{1-s}\Gamma(s+w+\frac{|n|}{2}).
\]
We see that if we identify $F_v$ with its image under complex conjugation, we may replace $\chi_v(z)$ with $\chi_v(\bar{z})$ and thus $n$ with $-n$, and $L_v(s,\chi_v)$ remains well-defined.  Finally, if $v_k$ is an archimedean place, we will write $\Gamma_k(s)$ for $\Gamma_\R(s)$ or $\Gamma_\C(s)$ depending on whether the completion $F_{v_k}$ is real or complex.

Take $\chi$ to be the inducing character for the principal series representation indexed by $\eta$ as in Definition \ref{ps} below. Then the normalising factor of the intertwining operator can be written as
\be
\label{meta}
m(\eta,s)=\frac{L(s,\chi)}{\epsilon(s,\chi,\psi)L(1+s,\chi)},
\ee
where $\psi$ is a fixed additive character of $F\backslash \A_F$. Finally, we shall refer to the \emph{non-trivial zeroes} of $L(s,\chi)$ as the zeroes $\rho$ of $L(s,\chi)$ such that $0\leq\mathrm{Re}(\rho)\leq1$.

Our first result expresses the integral of the logarithmic derivative of the intertwining operator as a sum over zeroes of Hecke $L$-functions. 

\begin{theorem}
\label{mainthm2}
Let $g(x)$ be a smooth compactly supported function on $\R_+^\times$, with $\hat{g}$ its Mellin transform and $g^*(x)=\overline{g(x^{-1})}$. Let $\chi$ be a Hecke character of $F$ of conductor $\f(\chi)$ associated to $\eta$, where at each archimedean place $v_k$, $\chi_{v_k}$ has ramification $w_k=a_k+ib_k$. Then the contribution of the normalising factor for a fixed $\eta$ to the spectral side of the trace formula for $G(\A)$,
\[
-\frac{1}{4\pi}\int_{-i\infty}^{i\infty} m(\eta,s)^{-1}m'(\eta,s)\hat{g}(s)ds 
\]
is equal to
\begin{align}
\label{mainthmformula}
& \frac12\sum_{\rho}(\hat{g}(\rho)+\hat{g}(-\bar\rho))+\frac{1}{2}\sum_{\mathfrak a} \Lambda({\mathfrak a})\big\{\chi({\mathfrak a})g(N({\mathfrak a}))+\bar\chi({\mathfrak a})g^*(N({\mathfrak a}))\big\}\notag\\
&+\log(N\f(\chi)|d_F|) g(1)\notag -\sum_{a_k\leq 0} \int_0^\infty g(x)(\delta_\chi+x^{w-1})dx \notag\\
&-\sum_{k=1}^{r_1+r_2}\frac{1}{2\pi}\int_{-\infty}^\infty \mathrm{Re}\Big[\frac{\Gamma'_k}{\Gamma_k}(\frac{1}{2}+it+w_k)\Big]\hat{g}(\frac{1}{2}+it)dt,
\end{align}
where $\delta_\chi$ is defined to be 1 if $\chi=1$ and 0 otherwise. The sums over $\rho$ aisre taken over the non-trivial zeroes of $L(s,\chi)$, and the sum over $\a$ is taken over integral ideals $\a\subset \O_F$. Moreover, the last term can be expressed as a function of $g(x)$ only, given in (\ref{gammaformula}). 
\end{theorem}

\noindent Note that the function $g$ above is obtained from a test function $f\in C_c^\infty(G(\A))$ by Lemma \ref{testfun}. In particular, we have identified the Mellin transform $\hat{g}(s)$ with the character of the induced representation $\tr\rho(f,\eta,s)$, which varies according to the complex parameter $s$. Also, we note that after the writing of this paper, Shahidi \cite{Sha} has introduced a new normalisation of the intertwining operator, which may simplify many of the terms in (\ref{mainthmformula}). We hope to return to this in a future study.

The form of \eqref{mainthmformula} reflects Weil's first explicit formula \cite{W52}. The expression for the archimedean contributions are obtained using work of Bombieri \cite{B}, so that the resulting expression is given solely in terms of the function $g$, which is useful for applications. Indeed, we shall view \eqref{mainthmformula} as a distribution on $C_c^\infty(\R^\times_+)$, which we shall denote by $D$. 

Our second result expresses $D$ as a sum of distributions, which we will refer to as \emph{Weil-type distributions} for their relation to Weil's second explicit formula \cite{W72}. This is particularly natural in the sense that the trace formula is an identity of distributions.

\begin{theorem}
\label{mainthm3}
The distribution $D(g)$ can be expressed as a sum of Weil-type distributions
\[
g(0)\log|d_F|+\frac12\int_{W_F}g(|w|)\chi(w)|w|^{-1}dw-\frac12\sum_v pv_0\int_{W_v}\hat{g}(|w|)\chi_v(w)\frac{|w|}{|1-w|}dw.
\]
The sum $v$ runs over all places of $F$, $d_F$ is discriminant of $F$, and $W_F, W_v$ are the global and local Weil groups of $F$ respectively. The finite part $pv_0$ is defined in \eqref{pv}. 
\end{theorem}

Having shown a relationship between the explicit formulas and the spectral side of the trace formula, one would hope that the connection will shed light on either of the two. As a first step, we show the following bound for the sums over zeroes, which follows as a straightforward corollary to Theorem \ref{mainthm2}. For simplicity, we shall define $D_1(g)$ by the relation
\[
D(g) = \frac12\sum_{\rho} (\hat{g}(\rho)+\hat{g}(-\bar\rho)) + D_1(g).
\]
Then we obtain a lower bound on the sum over zeroes of $L(s,\chi)$ in terms of \eqref{mainthmformula}. By the proof of Theorem \eqref{mainthm3}, the lower bound can also be given in terms of Weil-type distributions, but we prefer the explicit expression given below.

\begin{corollary}
\label{mainthm4}
Let $g=g_0*g_0^*$, for any $g_0$ in $C_c^\infty(\R_+^\times)$, and fix a central character $\omega$ of $G$. Then for any Hecke character $\chi\neq\omega$,
\be
\label{lowerbdhecke}
\frac12\sum_{\rho} (\hat{g}(\rho)+\hat{g}(-\bar\rho)) \geq - D_1(g) - D_2(g),
\ee
where $D_2(g)$ is the distribution arising from the normalised intertwining operator (\ref{LL3}). For $\zeta(s)$, we have
\begin{align}
\label{lowerbdzeta}
\sum_{\rho}\hat{g}(\rho)\geq 
&\int_{0}^\infty \Big\{g(x)+\frac{1}{4}g^*(x)\Big\}dx -\sum_{n=1}^\infty \Lambda(n)g(n)-\frac{1}{4\pi i}\int^\infty_{-\infty}m(it)\hat{g}(it)\frac{dt}{t}\notag\\
&-\int_1^\infty \Big\{g(x)+g^*(x)-\frac{2}{x}g(1)\Big\}\frac{x\ dx}{2(x^2-1)},
\end{align}
where $\rho$ runs over the non-trivial zeroes of $\zeta(s)$.
\end{corollary}

Note that there is implicitly a parameter $T>\sqrt3/2$ due to the truncation of Eisenstein series, though we do not explicate this here. For example, the right-hand side of (\ref{lowerbdzeta}) is evaluated at the point $T=1$. Allowing the parameter $T$ to vary, one may possibly optimise further the lower bound for different $g$. Also, we note that for applications to analytic number theory, it is useful to have the right-hand side be given solely in terms of $g$, and that the term containing $\hat{g}$ is equal to $\frac14m(0)g(0)+o(1)$ as $T$ tends to infinity (\cite[p.107]{K}). 

\subsection{Relation to other work}

It is well-known since \cite{JS2} that the spectral theory of Eisenstein series, and in particular the Maa\ss-Selberg relation, can be used to prove the nonvanishing of $L$-functions on the line Re$(s)=1$. In \cite{Sar}, Sarnak showed how to make this method effective, and in the spirit of de la Vall\'ee Poussin obtain standard zero-free regions for $\zeta(s)$. This was extended to other $L$-functions, for example \cite{GLa,GLi}. In our method, we are essentially averaging the Maa\ss-Selberg relation in the $t$-aspect, that is, integrating along the imaginary axis. Doing so, we get a hold of all possible zeroes in the critical strip.

The explicit formulas have been used recently to obtain striking results on low-lying zeroes of families of $L$-functions (see for example \cite{ILS,HM,FM,ST} and the references therein). In order to match the predictions of random matrix theory \cite{KS}, the compact support of the test function should be essentially arbitrary. Indeed, this expectation is similar to the desiderata for test functions satisfying Weil's criterion \cite{B,Bu}, which we discuss in Remark \ref{support}. At present, the support of test functions remains severely constrained. In Section \ref{sec5}, we describe a new approach to bounding the sum over zeroes, which, if made effective, might in certain cases allow for functions with larger support.

\subsection{Outline of paper} This paper is organised as follows: In the Section \ref{sec2}, we derive the explicit formula for Hecke $L$-functions following Bombieri \cite{B} for $\zeta(s)$. We also recall the adelic form of the explicit formula and the Weil criterion for Hecke $L$-functions in this context.  In Section \ref{sec3}, we recall the continuous contribution to the spectral side of the trace formula for $G$, following Langlands and Labesse \cite{LL}.  In Section \ref{sec4}, we relate the explicit formulas to the spectral side of the trace formula, proving Theorem \ref{mainthm2} and \ref{mainthm3} above. While the proof of Theorem \ref{mainthm3} essentially follows \cite{W72}, some care is needed because our normalisation differs by $\frac12$ at times because we take the representation theoretic point of view, where the centre of symmetry is the unitary axis rather than the line Re$(s)=\frac12$. Finally, in Section \ref{sec5} we discuss the application to lower bounds, giving Corollary \ref{mainthm4} as an example. 



\section{The Weil explicit formula}
\label{sec2}

In this section, we derive a version of Weil's explicit formula for Hecke $L$-functions in the sense of \cite{W52} adapted to the normalisation used in the trace formula, and derive the archimedean contributions in Theorem \ref{mainthm2}. We also state the second form as in \cite{W72} with reference to Theorem \ref{mainthm3}, and state Weil's criterion for the associated $L$-functions.

\subsection{First form of the explicit formula}

To state Weil's first explicit formula, fix the following notation: given $f(x)$ a complex-valued function in $C_c^\infty(\R_+^\times)$, define in this section only $f^*(x)=\frac{1}{x}f(\frac{1}{x})$ and the Mellin transform $\hat{f}(s)=\int_0^\infty f(x)x^{s-1} dx$. We will say that $f$ is even if $f=f^*$ and odd if $f=-f^*$.

We now state a version of Weil's explicit formula for $L(s,\chi)$. The shape of the distribution arising from the gamma factor is expressed differently than in \cite{W52}. In particular, we generalise the expression given by Bombieri \cite[p.186]{B} for $\zeta(s)$ to Hecke $L$-functions.

\begin{theorem}
\label{weilexpthm}
Let $g\in C_c^\infty(\R_+^\times)$, and $\chi$ a Hecke character of a number field $F$. Then we have
\begin{align}
\label{weilexp}
\sum_{\rho} \hat{g}(\rho)&=\delta_\chi\int_0^\infty (g(x)+g^*(x))dx-\sum_{\mathfrak a}\Lambda({\mathfrak a})\chi({\mathfrak a})(g(N({\mathfrak a}))+g^*(N({\mathfrak a})))\\
&-\log (|d_F|N\f(\chi))g(1)-\sum_{k=1}^{r_1+r_2}\frac{1}{2\pi}\int _{-\infty}^\infty 2\mathrm{Re}\Big[\frac{\Gamma'_k}{\Gamma_k}(\frac12+it)\Big]\hat{g}(\frac12+it)dt,\notag
\end{align}
where the sum $\rho$ is taken over non-trivial zeroes of $L(s,\chi)$. Moreover, the last sum can be expressed in terms of $g$, 
\begin{align}
\label{gammaformula}
\sum_{k=1}^{r_1+r_2}&\int_1^\infty \big\{g(x)+g^*(x)-\frac{M}{x^{M-a_k+ib_k}}g(1)\big\}\frac{x^{M-1-a_k+ib_k}}{x^M-1}dx\\
&+\Big(\gamma+(\frac{2}{M})\log( \frac{2\pi}{M})\Big)g(1),\notag
\end{align}
where $M=2$ if $F_{v_k}\simeq \R$ and $1$ if $F_{v_k}\simeq\C$.
\end{theorem}

\begin{proof}
We only discuss the archimedean contributions. The rest follows from standard derivations of the explicit formulas (e.g., \cite[Theorem 29]{Ing} and \cite{La}). Consider
\[
\sum_{k=1}^{r_1+r_2}\frac{1}{2\pi i }\int_{(c)}\frac{\Gamma'_k}{\Gamma_k}(s)\hat{g}(s)ds+ \sum_{k=1}^{r_1+r_2}\frac{1}{2\pi i }\int_{(c')}\frac{\Gamma'_k}{\Gamma_k}(1-s)\hat{g}(s)ds,
\]
for some $c>1,c'<0$. In order to obtain the explicit formula we compute the last two integrals as follows. First, we move the line of integration of both integrals to $c=c'=\frac12$, which we may do without encountering any pole of the integrand. Thus the sum of the two integrals becomes 
\[
\sum_{k=1}^{r_1+r_2}\frac{1}{2\pi}\int _{-\infty}^\infty 2\mathrm{Re}\Big[\frac{\Gamma'_k}{\Gamma_k}(\frac12+it)\Big]\hat{g}(\frac12+it)dt.
\]
Note that our $\Gamma_k(s)$ here depends on the ramification of $\chi$, and the $\Gamma_k(1-s)$ appearing in the second integral is associated to $\bar\chi$.

In the case $\Gamma_k(s)=\Gamma_\R(s)$, with ramification $w=a_k+ib_k$. (We will drop the subscripts $a=a_k$ and $b=b_k$ when the setting is clear.) We begin with $\frac{\Gamma_\R'}{\Gamma_\R}(s)=-\frac12\log\pi+\frac12\frac{\Gamma'}{\Gamma}\big(\frac{s}{2}\big),
$
so that we have
\be
-(\log \pi)g(1)+ \frac{1}{2\pi}\int _{-\infty}^\infty \mathrm{Re}\Big[\frac{\Gamma'}{\Gamma}(\frac{1}{2}(\frac12+it+a + ib))\Big]\hat{g}(\frac12+it)dt.
\label{bombgamm}
\ee
To treat the integral, we use the estimate in \cite[p.188]{B}, 
\be
\label{gamma'}
\frac{\Gamma'}{\Gamma}(z)=\log N -\sum_{n=0}^N \frac{1}{n+z}+ O\big(\frac{1+|z|}{N}\big),
\ee
which holds uniformly for Re$(z)>-\frac{N}{2}$ and $z$ not equal to zero or a negative integer. This gives
\[
\mathrm{Re}\Big[\frac{\Gamma'}{\Gamma}(\frac{1}{2}(\frac12+it+a + ib))\Big]=\log N -\sum_{n=0}^N \frac{2(2n+a+\frac12)}{(2n+a+\frac12)^2+(t+b)^2}+ O\big(\frac{1+|t+b|}{N}\big).
\]
And so the integral in (\ref{bombgamm}) becomes
\begin{align*}
&\frac{1}{2\pi}\int _{-\infty}^\infty \Big(\log N -\sum_{n=0}^N \frac{2(2n+a+\frac12)}{(2n+a+\frac12)^2+(t+b)^2}\Big)\hat{g}(\frac12+it)dt\\
&+ O\big(\int _{-\infty}^\infty\frac{1+|t+b|}{N}\Big|\hat{g}(\frac12+it)\Big|dt\big).\notag
\end{align*}
Since $\hat{g}$ is rapidly decreasing on any vertical line, the last integral converges and the error term is $O(1/N).$ Apply also Mellin inversion to the first term, whence 
\begin{align}
\label{bomgamm2}
&-\sum_{n=0}^N \frac{1}{2\pi}\int _{-\infty}^\infty \frac{2(2n+a+\frac12)}{(2n+a+\frac12)^2+(t+b)^2}\hat{g}(\frac12+it)dt+(\log N)f(1) + O\big(\frac{1}{N}\big).
\end{align}
We have by Fubini's theorem
\begin{align*}
&\frac{1}{2\pi}\int _{-\infty}^\infty \frac{2c}{c^2+(t+b)^2}\int_0^\infty g(x)x^{-\frac12+it}dx\ dt\\
&=\int_0^\infty g(x)x^{-\frac12} \frac{1}{2\pi i}\int _{-\infty}^\infty \Big(\frac{c}{t+ic}-\frac{c}{t-ic}\Big)x^{i(t-b)} dt\ dx,
\end{align*}
after making the change of variables $t\mapsto t-b$. Applying the calculus of residues yields 
\[
\frac{1}{2\pi i}\int _{-\infty}^\infty \Big(\frac{c}{t+ic}-\frac{c}{t-ic}\Big)x^{i(t-b)} dt=\min(x,x^{-1})^{c-ib}.
\]
Hence, taking $c=2n+a+\frac12$, we can write (\ref{bomgamm2}) as
which is equal to 
\[
-\int_1^\infty \Big(\sum_{n=0}^N x^{-2n-a+ib}\Big)(g(x)+g^*(x))\frac{dx}{x}+(\log N)f(1) + O\big(\frac{1}{N}\big).
\]
Finally, we can express the integral as 
\begin{align*}
& \int_1^\infty \Big(\sum_{n=0}^N x^{-2n-a+ib}\Big)\big(g(x)+g^*(x)-\frac{2g(1)}{x^{2-a+ib}}\big)\frac{dx}{x} +\sum_{m=1}^{N+1}\frac{1}{m}g(1)
\end{align*}
and substitute back to obtain
\begin{align*}
-&\int_1^\infty\frac{1-x^{-2N-2}}{1-x^{-2}}x^{-a+ib}(g(x)+g^*(x)-\frac{2g(1)}{x^{2-a+ib}}\big)\frac{dx}{x}+(\log N-\sum_{m=1}^{N+1}\frac{1}{m})f(1) + O\big(\frac{1}{N}\big).
\end{align*}
Now we take the limit as $N\to \infty$ and deduce for the integral in (\ref{bombgamm})
\[
-(\gamma+\log \pi)g(1)-\int_1^\infty \big\{g(x)+g^*(x)-\frac{2}{x^{2-a+ib}}g(1)\big\}\frac{x^{1-a+ib}}{x^2-1}dx,
\]
where $\gamma$ is the usual Euler-Mascheroni constant. This concludes the real archimedean case. The complex archimedean case $\Gamma_k(s)=\Gamma_\C(s)$ is almost identical to the first case, and is left to the reader.
\end{proof}
 
\subsection{Second form of the explicit formula}

 
Given a finite extension of number fields $K/F$, we may associate an $n$-dimensional representation $r$ of the relative Weil group of $K/F$, see \cite[p.9]{W72} for precise definitions. When we take the trace of this representation, we produce a character of $F$. The associated $L$-function is then referred to as the Artin-Hecke $L$-function. Define the functions $f_0(x)=\inf(x^\frac12,x^{-\frac12})$ and $f_1(x)=f_0(x)^{-1}-f_0(x)$ on $\R^\times_+$, and the \emph{principal value}
\be
\label{pv}
pv\int_0^\infty f(x)d^\times x=\lim_{t\to\infty}(\int^\infty_0(1-f_0(x)^{2t}f(x)d^\times x-2c\log t),
\ee
where $c$ is a constant such that $f(x)-cf_1(x)^{-1}$ is an integrable function on $\R^\times_+$. Furthermore, we denote for simplicity
\[
pv_0\int^\infty_0 f(x)d^\times x=pv\int_0^\infty f(x)+2c\log (2\pi).
\]
Then the second form of Weil's explicit formula is given as follows.

\begin{theorem}
[{\cite[p.18]{W72}}]
Let $L(s,\chi)$ be an Artin-Hecke $L$-function, and $g$ a smooth, compactly supported function on $\R_+^\times$. Then
\begin{align}
\label{weilexp2}
\sum_{\rho}\hat{g}(\rho)=&-g(1)\log|d_F|+\int_{W_F} g(|w|)\chi(w)(|w|^\frac12+|w|^{-\frac12})dw\\
&-\sum_v pv_0\int_{W_v}g(|w|)\chi_v(w)\frac{|w|^\frac12}{|1-w|}dw,\notag
\end{align}
where the sum over $\rho$ is taken over non-trivial zeroes of $L(s,\chi)$.
\end{theorem}

\subsection{The Weil criterion}

The novelty in Weil's method is the following condition, referred to as Weil's criterion. We will consider test functions in $C_c^\infty(\R_+^\times)$ which are multiplicative convolutions of a function $g$ and its transpose conjugate $\bar{g}^*$, hence
\[
g*\bar{g}^*=\int_0^\infty g(xy^{-1})\bar{g}^*(y)\frac{dy}{y}=\int_0^\infty g(xy)\overline{g(y)}dy. 
\]
Also note that the transform turns convolution into multiplication $\widehat{g*\bar{g}^*}(s)=\hat{g}(s)\hat{\bar{g}}^*(1-s).$
We can now generalise Bombieri's strengthening of Weil's criterion for $\zeta(s)$ \cite[p.191]{B} to Hecke $L$-functions:

\begin{theorem}
Let $W(g)$ be the linear functional defined by (\ref{weilexp}), on the space $C_c^\infty(\R_+^\times)$, so that
$W(g)=W(g^*)=\sum_{\rho} \hat{g}(\rho),
$
the sum taken over complex zeroes of $L(s,\chi)$ with $0<\mathrm{Re}(\rho)<1$. Then the Riemann hypothesis for $L(s,\chi)$ is equivalent to the statement that
$
W(g*\bar{g}^*)\geq 0 
$
with equality only if $g$ is identically zero. In short, we say that the functional is {\em positive-definite} on such functions.
\end{theorem}

\begin{proof}The proof of the theorem is now fairly routine, so we shall describe it only in brief. In the forward direction, assuming the Riemann hypothesis, is simple. For the converse statement, we refer to \cite{W52} or \cite[p.342]{La} which proceeds by contradiction. Namely, assuming that there exists a zero of $L(s,\chi)$ with real part different from $\frac12$, then using this zero to produce a test function $g$ on which the functional $W(g)$ is negative.
\end{proof}

\begin{remark}
\label{support}

The functional $W(g)$ was placed into a more tractable form by Barner, with the new conditions on the test functions later referred to as the Barner conditions \cite{La}. Moreover, one can formulate as we have an equivalent condition for smooth, compactly supported functions, and the criterion for $\zeta(s)$ is proved for functions whose support is restricted to a small neighborhood of 0. In particular, Yoshida proved the analogous criterion for smooth, compactly supported, even functions \cite{Yo}, and verified positivity for functions supported on $[-t,t]$ with $t=\log2/2$; Burnol proves positivity for $t=\sqrt 2$ \cite[Th\'eor\`eme 3.7]{Bu}, and Bombieri for $t=\log 2$ \cite[Theorem 12]{B}. One hopes that the lower bound obtained in Section \ref{sec5} might be used to extend the support of functions for which positivity holds.
\end{remark}


\section{The continuous contribution to the spectral side} 

\label{sec3}

In this section, we recall the contribution of the continuous spectrum to the spectral side of the trace formula for $G=SL(2)$. For details, we refer the reader to \cite[\S5]{LL}.

Consider the adelic quotient
$
G(F)\backslash G(\A)/\prod_{v}'K_v,
$
the product taken over all places $v$ of $F$, and $K_v$ are maximal compact subgroups, taken to be $G(\mathcal O_{F_v})$ for almost all $v$.
\begin{definition}
\label{ps}
We first define the principal series for $G(\A)$. Let $A$ be the group of diagonal matrices in $G(\A)$ and let $A_F=A\cap GL_2(F)$. We consider the set $D^0$ of characters of $\eta$ of $A_F\backslash A$ such that 
$
\eta\big|_{Z}=\omega^{-1},
$
and each $\eta$ is again defined by the pair $(\mu,\nu)$ of idele class characters on $F^\times\backslash \A^\times$. Here $\omega$ is a fixed character of the center $Z$ of $G$. We also have the analogous height function
\be
\label{A'}
H:\begin{pmatrix}a & 0 \\ 0 & a^{-1}\end{pmatrix}\mapsto |a|^2,
\ee
where $|\cdot|$ is taken to be the adelic norm. Putting these together we obtain the representation $\rho(g,\eta,s)$ acting by right translation on the induced representation space $\text{Ind}_{NA}^{G} 1_N\otimes (\eta\otimes H^\frac{1+s}{2}).$ The resulting space is that of smooth functions $\varphi_s$ on $N(\A)\backslash G(\A)$ satisfying
\[
\varphi_s(\begin{pmatrix}a & * \\ 0 & a^{-1}\end{pmatrix}k)=\mu\nu^{-1}(a)|a|^{1+s}\varphi(k),
\]
where $k$ is an element of $K=\prod' K_v$. By the Iwasawa decomposition this space of functions can be identified with those on ${K}$. Moreover, since $A\backslash G={A}_{GL(2)}\backslash GL(2),$ with ${A}_{GL(2)}$ the diagonal matrices in $GL(2)$,  we may regard the space of functions on which $\rho(g,\eta,s)$ acts as a space of functions on $GL(2)$ by extending $\eta$ trivially to $A_{GL(2)}$. Finally, we set $\mu\nu^{-1}=\chi$.
\end{definition}

If $\eta$ is associated to the pair $(\mu,\nu)$, then the \emph{intertwining operator}
\[
(M(\eta,s)\varphi)(g)=\prod_{v}M(\eta_v)=\prod_v\int_{N_v}\varphi(wn_vg_v)dn_v
\]
intertwines the principal series of $(\mu,\nu)$ with that of $(\nu,\mu)$. We normalise it as
\[
M(\eta,s)=\frac{L(1-s,\chi^{-1})}{L(1+s,\chi)}\otimes_v R(\eta_v,s):=m(\eta,s)\otimes_v R(\eta_v,s),
\]
where $R(\eta_v)$ denotes the normalised local intertwining operator. We will refer to the quotient of completed $L$-functions as the normalising factor $m(\eta,s)$. By the functional equation of the Hecke $L$-function, we obtain the form of $m(\eta,s)$ given in (\ref{meta}). Note that $M(\eta,s)$ is essentially the scattering or Eisenstein matrix that is given explicitly, for example, in \cite{Mas}, and $m(\eta,s)$ its determinant.

Now consider the regular representation $\rho$ of $G(\A)$ acting on the discrete spectrum $L^2_\text{disc}(G(F)\backslash G(\A),\omega)$ of $L^2(G(F)\backslash G(\A),\omega)$ the space of square integrable functions on $G(F)\backslash G(\A)$ which transform according to the central character $\omega$. Choose a smooth function $f=\prod f_v$ in $G(\A)$ that is compact modulo $Z(\A)$ such that $f(zg)=\omega(z)f(g)$ for any $z$ in $Z$, and for almost all $v$, $f_v$ is supported on $G(F)\cap GL_2({\mathcal O}_F)$. Define a convolution operator $\rho_0(f)$ acting on $L^2_\text{disc}(G(F)\backslash G(\A),\omega)$ by
\[
\rho_0(f)(\phi(x))=\int_{Z_(\A)\backslash G(\A)}f(g)\rho_0(g)\phi(x)dg.
\]
It is a Hilbert-Schmidt operator, and in particular trace class.

\begin{remark}
\label{pos}
Note that functions that are convolutions of the form 
$
f(x)=f_0(x)*\overline{f_0(x^{-1})},
$
are positive definite, where $f_0\in C_c^\infty(G)$, so that for such functions the operator $\rho(f)$ is self-adjoint and positive definite \cite[\S2.4]{GGPS}, and as a consequence its restriction to any invariant subspace is also positive definite. This property shall be useful to us when studying the explicit formulas.
\end{remark}

The trace formula now expresses the trace $\tr(\rho_0(f))$ in two ways, first as a sum of characters of representations, and second as a sum of orbital integrals. We examine the terms in the spectral side of the trace formula arising from the noncuspidal spectrum of $L^2(G(F)\backslash G(\A),\omega)$. For simplicity, we shall assume the central character $\omega$ is trivial. We then recall the following result from Labesse and Langlands.

\begin{theorem}
\label{LL}
Let $f$ and $\rho$ be defined as above. Then the contribution of the continuous spectrum to the spectral side of the trace formula for $G(\A)$ obtained by specialising the terms (5.5) and (5.6) in \cite[p.754]{LL} is
\begin{align}
&\sum_\eta \frac{1}{4\pi}\int_{-i\infty}^{i\infty} m(\eta,s)^{-1}m'(\eta,s)\tr(\rho(f,\eta,s))|ds| + \sum_{\eta^2=\omega}-\frac{1}{4}\tr (M(\eta,0)\rho(f,\eta,0))\label{LL2}\\
&+\sum_\eta\sum_v\frac{1}{4\pi}\int_{-i\infty}^{i\infty}\tr({R^{-1}(\eta_v,s)R'(\eta_v,s)\rho(f_v,\eta_v,s))\prod_{w\neq v}\tr(\rho(f_w,\eta_w,s))|ds|}. \label{LL3}
\end{align}
Here the summation over $\eta$ runs over characters in $D^0$, and $R^{-1}(\eta_v,s)$ is the inverse operator of $R(\eta_v,s)$.
\end{theorem}

\noindent We will mainly be interested in the first term in (\ref{LL2}), which involves the logarithmic derivative of $L$-functions.

\subsection{Test functions}

The following lemma shows that the character of the induced representation $\rho(f,\eta,s)$ defines the Mellin transform of a smooth compactly supported function on $\R^\times_+$.

\begin{lemma}[Test function]
\label{testfun}
Let $f$ be a function in $C_c^\infty(Z(\A)\backslash G(\A))$. Then the character $\tr(\rho(f,\eta,s))$ is the Mellin transform of a function $g$ in $C_c^\infty(\R^\times_+)$.
\end{lemma}

\begin{proof}
Using the Iwasawa decompostion, the trace $\tr(\rho(f,\eta,s))$ is equal to
\[
\int_{K}\int_{\R^\times_+}\int_{(\A^\times)^1}\int_{N_\A} f(k^{-1}a^{-1}na\begin{pmatrix}t&0\\0&t^{-1}\end{pmatrix}k)|t|^{1+s}dn\ da\ d^\times t\ dk.
\]
Now we may interchange the following,
\[
na=\begin{pmatrix}1&n\\0&1\end{pmatrix}\begin{pmatrix}a&0\\0&b\end{pmatrix}=\begin{pmatrix}a&0\\0&b\end{pmatrix}\begin{pmatrix}1&na^{-1}b\\0&1\end{pmatrix}=an',
\]
which multiplies the measure on $N_\A$ by an element of norm 1, thus preserving the measure. Doing so, we obtain
\[
\int_K\int_{\R^\times_+}\int_{N_\A} f(k^{-1}n\begin{pmatrix}t&0\\0&t^{-1}\end{pmatrix}k)|t|^{1+s}dn\ d^\times t\ dk.
\]
Now denote the integration over the compact set $K$ by an auxiliary function $\Phi(g):=\int_{K} f(k^{-1}g k)dk,$ which allows us to write
\[
\int_{\R^\times_+}\int_{N_\A} \Phi(n\begin{pmatrix}t&0\\0&1\end{pmatrix})|t|^{1+s}dn\ d^\times t = \int_{\R^\times_+}S\Phi(t)|t|^{s} d^\times t.
\]
Here $S$ is the Satake transform
\[
S\Phi(t)=\prod_vH_v(t_v)^\frac{1}{2}\int_{N_v}\Phi(n_v\begin{pmatrix}t_v&0\\0&t^{-1}\end{pmatrix})dn_v,
\]
with $H_v(t_v)^\frac12=|t_v|$ the usual modulus character of $P=M\ltimes N$, related to the height function $H$ defined in (\ref{A'}).  Note that while we recognise the integral as the Satake transform, we are not in fact using the Satake isomorphism at unramified places. Finally, setting $g=S\Phi$, we arrive at the Mellin transform
\[
\hat{g}(s)=\int_{\R^\times_+} g(t)|t|^{s} d^\times t=\int_0^\infty g(t) t^{s-1}dt.
\]
Though for our purposes it is enough to know that the latter measure is nonzero and finite, which is certainly true. (This is related to the constant $c$ in \cite[{\S16}]{JL}.) So given a test function $f$, we may view the trace $\tr(\rho(f,\eta,s))$ as the Mellin transform $\hat{g}(s)$ of a smooth compactly supported function defined on the positive real numbers.
\end{proof}

It turns out that all functions $g$ in $C_c^\infty(\R_+^\times)$ can be obtained from $f$ in this manner. Taking $K=SO_2(\R)$, given an element $f$ of the spherical Hecke algebra $C_c^\infty(K\backslash G(\R)/K)$ we have the Harish transform
\[
H(a)^\frac12\int_Nf(an)dn=|H(a)^\frac12-H(a)^{-\frac12}|\int_{A\backslash G}f(x^{-1}ax)dx,
\]
where $A$ is the Cartan subgroup of $G(\R)$. The transform is invariant under the action of the Weyl group $W$ of $G$, and $A/W$ can be represented by matrices $a$ in $A$ such that $H(a)\geq1$. Thus changing variables $y=e^{u}$ so that $H(a)=e^{2u}$, we may write additively
\[
g(u)=|e^{u}-e^{-{u}}|\int_{A\backslash G(\R)}f(x^{-1}\begin{pmatrix}e^{u}&0\\0&e^{-{u}}\end{pmatrix}x)dx,
\]
and observe that $g(u)=g(-u)$. The Harish transform is an algebra isomorphism from the space $C_c^\infty(K\backslash G(\R)/K)$ to $C_c^\infty(A)^W$, where the superscript indicates functions invariant under the Weyl group $W$, and the product given by convolution. Over a $p$-adic field, its analog is the Satake isomorphism.

Also, we note image of the Mellin transform defined for $g$ in $C_c^\infty(A)^W$, 
\[
\hat{g}(s)=\int_A g(a)H(a)^\frac{s}{2}d^\times a,
\]
lies in the Paley-Wiener space, consisting of entire functions $f$ for which there exists positive constants $C$ and $N$ such that
$
|f(x+iy)|\ll C^{|x|}(1+|y|)^{-N},
$
which is to say $f$ has at most exponential growth with respect to $x$ and is uniformly rapidly decreasing in vertical strips.

\section{Explicit formulas for the spectral side}

\label{sec4}

\subsection{Sums of zeroes in the continuous spectral terms}

To illustrate the method, we will first consider the classical trace formula for $G(\R)$ without ramification. The proof in the adelic setting will be similar. In particular, we consider $\rho_0$ the regular representation of $G(\R)$ on the Hilbert space $L^2(G(\Z)\backslash G(\R) / SO_2)$. In this case, the only character $\eta$ that appears is the trivial one, and the intertwining operator is $m(s)=\xi(s)/\xi(1+s)$, with $\xi(s)$ the completed Riemann zeta function.

\begin{theorem}
\label{mainthm1}
Let $g(x)$ be a smooth compactly supported function on $\R_+^\times$, with $\hat{g}$ its Mellin transform. The contribution of the normalising factor
\[
-\frac{1}{4\pi}\int^\infty_{-\infty}\frac{m'}{m}(ir)\hat{g}(ir)dr 
\]
to the spectral side of the trace formula for $G(\R)$ is given by
\begin{align*}
\sum_{\rho}&\hat{g}(\rho)-\int_{0}^\infty \Big\{g(x)+\frac{1}{4}g^*(x)\Big\}dx +\sum_{n=1}^\infty \Lambda(n)g(n) \\
&+\int_1^\infty \Big\{g(x)+g^*(x)-\frac{2}{x}g(1)\Big\}\frac{x\ dx}{2(x^2-1)}+\frac12(\log4\pi-\gamma)g(1),
\end{align*}
where the sum $\rho$ runs over non-trivial zeroes of $\zeta(s)$.
\end{theorem}

\begin{proof}
By the functional equation, we express the logarithmic derivative of the normalising factor as
\[
\frac{m'}{m}(s)=\frac12\frac{\Gamma'}{\Gamma}\Big(\frac{s}{2}\Big)+\frac{\zeta'}{\zeta}(s)-\log\pi+\frac12\frac{\Gamma'}{\Gamma}\Big(-\frac{s}{2}\Big)+\frac{\zeta'}{\zeta}(-s).
\]
Then we substitute this into our integral above, whence
\[
-\frac{1}{4\pi  i}\int^{i\infty}_{-i\infty}\Big\{\mathrm{Re}\frac{\Gamma'}{\Gamma}\Big(\frac{s}{2}\Big)+2\mathrm{Re}\frac{\zeta'}{\zeta}(s)-\log\pi\Big\}\hat{g}(s) dr,
\]
using the property that $\hat{g}(ir)$ is even. Applying Mellin inversion to the last term, we obtain
\be
\label{ingh}
\frac12 g(1) \log \pi -\frac{1}{4\pi  i}\int^{i\infty}_{-i\infty}\mathrm{Re}\Big[\frac{\Gamma'}{\Gamma}\Big(\frac{s}{2}\Big)\Big]\hat{g}(s)ds-\frac{1}{2\pi  i}\int^{i\infty}_{-i\infty}\mathrm{Re}\Big[\frac{\zeta'}{\zeta}(s)\Big]\hat{g}(s) ds.
\ee

To treat the third term, move the line of integration to the right to $(c-i\infty,c+i\infty)$ for some $c>1.$ This step is justified by integrating over a rectangle $R$ with vertices $(c\pm iT,\pm iT)$ and showing that the integral over the horizontal edges tends to 0 as we let $T$ tend to infinity along a well-chosen sequence $(T_m), m=2,3,\dots$ such that $m<T_m<m+1$ and $\zeta'/\zeta(s)=O(\log^2 T)$
for any $s$ with $-1\leq \sigma\leq 2$ and $T=T_m$, and as $T$ tends to infinity (see \cite[Theorem 26, p.71]{Ing}). Allowing then $t$ to tend to infinity via the sequence $(T_m)$, the contribution from the horizontal edges vanish, as the transform $\hat{g}(s)$ of a smooth, compactly supported function has rapid decay along a fixed vertical strip. Moving the line of integration to the left, we obtain the residues of $\zeta'/\zeta(s)$ due to the zeroes of $\zeta(s)$ inside the critical strip $0< \sigma < 1,$ and the simple pole of $\zeta(s)$ at $s=1$. Then allowing $T$ to go to infinity we obtain
\[
-\frac{1}{2\pi  i}\int^{i\infty}_{-i\infty}\frac{\zeta'}{\zeta}(s)\hat{g}(s) ds=\sum_{\rho}\hat{g}(\rho)-\hat{g}(1)-\frac{1}{2\pi i}\int_{(c)}\frac{\zeta'}{\zeta}(s)\hat{g}(s)ds,
\]
where the sum $\rho=\beta+i\gamma$ runs over the zeroes of $\zeta(s)$ with $0<\beta<1$. The integral now being in the region of absolute convergence, we may integrate term by term and apply Mellin inversion to get
\be
\label{ingh2}
\sum_{\rho}\hat{g}(\rho)-\hat{g}(1)+\sum_{n=1}^\infty \Lambda(n)g(n)+\frac12 g(1) \log \pi -\frac{1}{4\pi  i}\int^{i\infty}_{-i\infty}\mathrm{Re}\Big[\frac{\Gamma'}{\Gamma}\Big(\frac{s}{2}\Big)\Big]\hat{g}(s)ds.
\ee

To treat the last integral, we first observe that
$
\frac{\Gamma'}{\Gamma}(s)=\log s-\frac{1}{2s}+O({|s|^{-2}})
$
uniformly in any fixed angle $|\text{arg}(s)|<\pi$ as $|s|$ tends to infinity. Now we move the line integration to the line Re$(s)=\frac12$, 
\[
-\frac{1}{4\pi  i}\int^{i\infty}_{-i\infty}\mathrm{Re}\Big[\frac{\Gamma'}{\Gamma}\Big(\frac{s}{2}\Big)\Big]\hat{g}(s)ds=-\frac14\hat{g}(0)-\frac{1}{4\pi  i}\int_{(\frac12)}\mathrm{Re}\Big[\frac{\Gamma'}{\Gamma}\Big(\frac{s}{2}\Big)\Big]\hat{g}(s)ds.
\]
Here one half the residue of the pole of $\Gamma(s)$ at $s=0$ is obtained as the initial line of integration is over $s=0$. Then from the proof of Theorem \ref{weilexpthm}, with $M=2$ and $a,b=0$ we can rewrite this as
\[
-\frac14\int_{0}^\infty g^*(x)dx+\frac12(\log 4+\gamma)g(1)+\int_1^\infty \Big\{g(x)+g^*(x)-\frac{2}{x}g(1)\Big\}\frac{x\ dx}{2(x^2-1)},
\]
(see also \cite[p.190]{B}). Then substituting this last expression into (\ref{ingh2}) proves the claim.

\end{proof}

Thus we see that the sum over zeroes as in Weil's explicit formula appear in the continuous spectral terms of the trace formula for $G(\R)$. Having treated the basic case, the general case follows easily.

\begin{proof}[Proof of Theorem \ref{mainthm2}]
The method of proof is similar to that of the previous theorem. By the functional equation, the integral becomes
\[
\label{adexp}
\frac{1}{4\pi i}\int_{-i\infty}^{i\infty} \Big\{\frac{\epsilon'}{\epsilon}(-s,\bar{\chi},\psi)+\frac{\epsilon'}{\epsilon}(s,\chi,\psi)-\frac{L'}{L}(s,\chi)-\frac{L'}{L}(-s,\bar{\chi})\Big\}\hat{g}(s)ds,
\]
where $\hat{g}(s)=\tr(\rho(f,\chi,s))$. We first treat the epsilon factors. Recall that
$
\epsilon(s,\chi,\psi)=W(\chi)|N(\f(\chi))d_F|^{s-\frac12}
$
where the global Artin conductor $\f(\chi)=\prod p_v^{\f_v(\chi)}$ is a product over all primes $p_v$ which ramify in $F$, of local conductors $\f_v(\chi)$ of the local character $\chi_v$. Then
\begin{align}
&\frac{1}{4\pi i}\int_{-i\infty}^{i\infty} \Big\{\frac{\epsilon'}{\epsilon}(-s,\bar{\chi},\psi)+\frac{\epsilon'}{\epsilon}(s,\chi,\psi)\Big\}\hat{g}(s)ds=\log(N\f(\chi)|d_F|) g(1),
\label{dF}
\end{align}
since $\f(\chi)=\f(\bar\chi)$. Note that the additive character $\psi$ is chosen as usual to give the self-dual measure on $F$. 

Now we turn to the $L$-functions. Denoting by $w$ the ramification of $\chi_{v_k}$ at archimedean places, the logarithmic derivative is
\[
\frac{L'}{L}(s,\chi)+\frac{L'}{L}(-s,\bar\chi)=\sum_{k=1}^{r_1+r_2}\Big\{\frac{\Gamma'_k}{\Gamma_k}(s+w_k)+\frac{\Gamma'_k}{\Gamma_k}(-s+\bar{w}_k)\Big\}+\frac{\zeta_F'}{\zeta_F}(s,\chi)+\frac{\zeta_F'}{\zeta_F}(-s,\bar\chi).
\]
We then separate the integral, into the logarithmic derivatives of $\zeta_F(s)$ and $\Gamma_k(s)$ respectively. Following the proof of Theorem \ref{weilexpthm} and Theorem \ref{mainthm1}, we obtain 
\begin{align}
\label{zetas2}
&\frac12\sum_{\rho}(\hat{g}(\rho)+\hat{g}(-\bar\rho))-\delta_\chi\int_0^\infty g(x)dx  \notag\\ 
&+\frac{1}{2}\sum_{\mathfrak a} \Lambda({\mathfrak a})\big\{\chi({\mathfrak a})g(N({\mathfrak a}))+\bar\chi({\mathfrak a})g^*(N({\mathfrak a}))\big\},
\end{align}
where as before $\delta_\chi$ is 1 if $\chi$ is trivial and zero otherwise. Here we have used the fact that if $\rho$ is a zero of $L(s,\chi)$, then by the functional equation $\bar{\rho}$ is a zero of $L(s,\bar\chi)$.

Finally, for the gamma factors, we begin with the expression
\begin{align*}
&-\sum_{k=1}^{r_1+r_2}\frac{1}{2\pi i}\int_{-i\infty}^{i\infty} \mathrm{Re}\Big[\frac{\Gamma'_k}{\Gamma_k}(s+w_k)\Big]\hat{g}(s)ds,
\end{align*}
since $s$ here is pure imaginary. We now move the line of integration to Re$(s)=\frac12$, encountering residues of poles of the gamma functions in the following scenarios: Recall that if $w=a+ib$, then for real places we have $a=0$ or 1, whereas for complex places we have $a\in\Z[\frac12]$. Since the poles of $\Gamma(s)$ occur at $s=0,-1,-2,\dots$, we see that the integral of $\Gamma'_k/\Gamma_k(s+w_k)$ passes through a pole only if $a\leq 0$. We thus have
\begin{align}
&-\sum_{a_k\leq 0} \int_0^\infty g(x)x^{w-1}dx -\sum_{k=1}^{r_1+r_2}\frac{1}{2\pi}\int_{-\infty}^\infty \mathrm{Re}\Big[\frac{\Gamma'_k}{\Gamma_k}(\frac{1}{2}+it+w_k)\Big]\hat{g}(\frac{1}{2}+it)dt,\notag
\end{align}
and the integral in this last expression is exactly the one obtained in Theorem \ref{weilexpthm}. Then putting these together with (\ref{dF}) and (\ref{zetas2}) proves the claim.
\end{proof}

\subsection{Weil distributions in the continuous spectral terms}

Having shown that the sum over zeroes appear in the contribution of the continuous spectrum to the trace formula, we relate this to the distributions arising in Weil's explicit formula \cite[p.18]{W72}. We will say the resulting distributions are of {\em Weil-type}, for their resemblance to those appearing in the Weil formula (\ref{weilexp2}); in our theorem there appears a difference of certain exponents of $\frac12$, this is due to our choice of normalisation. The rest of this section will be devoted to the proof of Theorem \ref{mainthm3}. We first collect several lemmas, from which the theorem will follow quickly.

\begin{lemma}
\label{weil1}
Let $g$ be a function in $C_c^\infty(\R^\times_+)$, and $\hat{g}$ its Mellin transform. Assume that there exists an $A>\frac12$ such that $g(t)=O(t^{A})$ as $t$ tends to 0 and $g(t)=O(t^{-A})$ as $t$ tends to infinity. Then for any $\sigma$ such that $|\sigma-\frac12|<A$, the following formula holds by Mellin inversion:
\[
\int_{\R^\times_+} \hat{g}(\sigma-\frac12+it)X^{\sigma+it}dt=2\pi X^\frac{1}{2}g(X).
\]
\end{lemma}

\begin{proof}
The growth assumption on $g$ implies that the transform $\hat{g}(\sigma-\frac12+it)$ is holomorphic in the region $|\sigma-\frac12|<A$, so that the inversion formula is independent of $\sigma$ in this range.
\end{proof}


\begin{lemma}
\label{weil2}
Let $v$ be a nonarchimedean completion of $F$, and $q_v$ the cardinality of the residue field of $F_v$. Then
\[
\frac{d}{ds}\log L_v(s,\chi) = -\log q_v\sum_{n=1}^\infty q_v^{-ns}\int_{W^0_v}	\chi_v(\f^n	w_0)dw_0,
\]
where $\f$ is a Frobenius element in $W_{F_v}$, and Re$(s)>1$.
\end{lemma}

\begin{proof}
We first justify interchanging the sum and integral: the character $\eta_v$ is assumed to be unitary, that is $|\chi(w_0)|\leq1$ for all $w_0$ in $W^0_{F_v}$, so the sum converges absolutely for Re$(s)>1$, then apply Proposition 1 of {\cite[p.12]{W72}}
\end{proof}





\begin{lemma}
\label{gamma0}The integral 
\be
\frac{1}{2\pi i}\int \hat{g}(s-\frac12)d\log\Gamma_k(\frac{1}{2}+s+a+ib)+\hat{g}(\frac12-s)d\log\Gamma_k(\frac{1}{2}+s+a-ib)\label{gamma1}\ee
with $a\geq0, b\in \R$, taken over the line Re$(s)=\frac12$, can be expressed as
\be
-pv_0\int_0^\infty g(\nu)\nu^{\frac12+ ib}\frac{f_0(\nu)^{2a+1-E}}{f_1(\nu^E)}d^\times \nu\label{gamma2}\ee
where $E=\dim_{F_{v_k}}\C$, which is to say $E=2$ if $F_{v_k}\simeq\R$ and $E=1$ if $F_{v_k}\simeq \C$.
\end{lemma}

\begin{proof}
A detailed proof of the analogous statement in \cite{W72} is supplied in \cite[p.162-174]{Mor}. We sketch the proof, indicating the necessary modifications. The transform $\hat{h}(s)$ used in \cite{W72,Mor} is a shifted Mellin transform:
\[
\hat{h}(s)=\int_0^\infty h(\nu)\nu ^{\frac12-s}d^\times\nu,
\]
which relates to our function $\hat{g}(s)$ by the relation
$\hat{h}(s)=\hat{g}(\frac12-s)=\hat{g}(s-\frac12).
$
The integral (\ref{gamma1}) is
\[
\frac{1}{2\pi }\int_{-\infty}^\infty \hat{g}(ir)\frac{\Gamma'_k}{\Gamma_k}(1+ir+a+ib)+\hat{g}(-ir))\frac{\Gamma'_k}{\Gamma_k}(1+ir+a-ib)\ dr.
\]

We consider the two cases, first where $k=\R$ the integral becomes after a change of variables $r\mapsto r-b$:
\[
-g(1)\log \pi +\frac{1}{4\pi}\int_{-\infty}^\infty \hat{g}(i(r-b))(\frac{\Gamma'}{\Gamma}(\frac{1+ir+a}{2})+\frac{\Gamma'}{\Gamma}(\frac{1-ir+a}{2})dr,
\]
where the first term follows from Mellin inversion. Considering next the integral, we shift the contour from $ir$ to $-\frac12+ir$, noting that $\hat{g}(s)$ and $\Gamma'/\Gamma(s)$ are holomorphic in this range. As a result we obtain
\[
\frac{1}{4\pi}\int_{-\infty}^\infty \hat{g}(-\frac12+i(r-b))(\frac{\Gamma'}{\Gamma}(\frac{\frac12+ir+a}{2})+\frac{\Gamma'}{\Gamma}(\frac{\frac12-ir+a}{2})dr.
\]
For $s$ satisfying \textnormal{Re}$(s)>0$, we have by \cite[p.160-162]{Mor},
\[
-2\frac{\Gamma'}{\Gamma}(s)=pv\int_0^\infty\frac{f_0(x)^{2s-1}}{f_1(x)}d^\times x,
\]
and it follows then that
\[
\frac{\Gamma'}{\Gamma}(\frac{\frac12+ir+a}{2})+\frac{\Gamma'}{\Gamma}(\frac{\frac12-ir+a}{2})=-pv\int_0^\infty \frac{f_0(\nu)^{-\frac12+a}}{f_1(\nu)}\nu^{ir/2}d^\times\nu.
\]
Making the change of variables $\nu\mapsto \nu ^2$, the integral then becomes after interchanging the order of integration:
\[
-pv\int_0^\infty \frac{f_0(\nu)^{2a-1}}{f_1(\nu^2)}\cdot \frac{1}{2\pi}\int_{-\infty}^\infty\hat{g}(-\frac12+i(r-b))\nu^{ir} dr\ d^\times\nu.
\]
Then applying Lemma \ref{weil1} to the inner integral, we have
\[
\frac{1}{2\pi}\int_{-\infty}^\infty\hat{g}(-\frac12+i(r-b))\nu^{ir} dr =\nu^{\frac12+ib}g(\nu).
\]
Putting it together we obtain the desired form of (\ref{gamma2}). 
The second case where $k=\C$ follows similarly.

\end{proof}

\begin{corollary}
[Archimedean contribution]
Let $\chi_v$ be a character of the local archimedean Weil group $W_v$, then  the expression (\ref{gamma2}) can be rewritten as
\[
-pv_0\int_{W_v}g(|w|)\chi_v(w)\frac{|w|_v}{|1-w|_v}dw\label{gamma3}.
\]
\end{corollary}

\begin{proof}
The proof of this follows immediately from (\ref{gamma2}) by Lemma 3 of \cite[p.165]{Mor}, with the function $\varphi$ of the lemma replaced with
$\varphi(w)=\frac{|w|_v}{|1-w|_v}.$
In particular, in the numerator we have $|w|_v$ instead of $|w_v|^\frac12$.
\end{proof}

\begin{lemma}[Contribution of the conductor]
\label{cond}
Consider the Herbrand distribution $H_v$ on the local Weil groups $W_v$, which is described in \cite[Ch.VIII, \S3, XII, \S4]{We1} and \cite[Ch.II, \S6]{Mor}. It is given by
\[
H_v(\chi)=\int_{W_v^0}\chi_v(w_0)dH_v(w_0)
\]
where $\chi_v$ is a character of the restriction to $W_v$ of a unitary representation $\chi$ of the Weil group $W_{F}$. The contribution of the conductor can be expressed as
\[
g(1)\log|N\f(\chi)|=-\sum_v\log q_v \int_{W_v^0} g(|w_0|)\chi_v(x)dH_v(w_0).
\]
\end{lemma}

\begin{proof}
We follow \cite[p.17]{W72} and \cite[p.158]{Mor}.  The local Artin conductor $\f_v(\chi)$ can be written as the integral 
$\int_{W_v^0}\chi(x)dH_v(x),
$
so that the contribution from the conductor for each place $v$ where $p_v$ ramifies is
\[
\log p_v\int_{W_v^0}h(|w_0|)\chi(w_0)dH_v(w_0),
\]
and zero otherwise.   By the product formula for the conductor, the contribution of the conductor is then
\begin{align*}
g(1)\log|N\f(\chi)|
&=-\sum_v \log q_v \int_{W_v^0}g(|w_0|)\chi_v(w_0)dH_v(w_0).
\end{align*}
\end{proof}
Using this expression we are able to express the archimedean and nonarchimedean integrals in a uniform manner. Let $g$ be a locally constant function on $W_v$, and suppose that the integral of
$
\frac{g(w)}{|1-w|}
$
over ${W-W_v^0}$ exists. Then define the principal value by
\[
pv_0\int_{W_v}\frac{g(w)}{|1-w|}dw=\int_{W_v^0}\frac{g(w)-g(1)}{|1-w|}dw+\int_{W-W_v^0}\frac{g(w)}{|1-w|}dw.
\]
The following is given in \cite[p.18]{W72} (see also \cite[Part I, Chap. VIII]{Mor2}).

\begin{corollary}
[Nonarchimedean contribution]
\label{nonarch}
The contribution of the nonarchimedean integrals and the conductor can be combined to obtain
\[
pv_0\int_{W_v}\hat{g}(|w|)\chi_v(w)\frac{|w|}{|1-w|}dw,
\]
where $1-w$ and $|1-w|$ are to have the same sense by the embedding of $W_v$ into the division algebra $A_v$.
\end{corollary}

Finally, we put all these together to prove the main theorem:

\begin{proof}[Proof of Theorem \ref{mainthm3}]
First, from the logarithmic derivative of $m(\eta,s)$ as in (\ref{adexp}), apply Lemma \ref{cond} to the contribution of the epsilon factors (\ref{dF}),
\begin{align*}
&g(1)\log|d_F|+g(1)\log|N\f(\eta)|\\
&=g(1)\log|d_F|-\sum_v \log q_v \int_{W_v^0}g(|w_0|)\chi_v(w_0)dH_v(w_0).
\end{align*}
Second, consider the individual summands involving the gamma factors, which converge as $\prod \hat{g}_k(r)$ has rapid decay at infinity. They take the form
\[
\frac{1}{4\pi }\int^{i\infty}_{-i\infty}\Big(\frac{\Gamma'_k}{\Gamma_k}(1+ir+w_k)+\frac{\Gamma_k'}{\Gamma_k}(1-ir+\bar{w})\Big)\hat{g}(ir)dr.
\]
We use the property of $\hat{g}$ as an even function, and apply Lemma \ref{gamma0}, with $w_k=a_k+ib_k$, to obtain
\[
-\frac12pv_0\int_{W_v}\hat{g}(|w|)\chi_v(w)\frac{|w|}{|1-w|_v}dw,
\]
where the constants $a,b$ depending on the ramification are accounted for by the local character $\chi_v$. Third, for the $\zeta_F(s,\chi)$ terms, write $L_v(s,\chi)$ for the local $L$-factor in the Euler product for $\zeta_F(s,\chi)$. For absolute convergence, we shift the contour slightly to the right of Re$(s)=1$, where we may use the Euler product expansion for the logarithmic derivative:
\[
\frac{1}{4\pi i}\int^{1+i\infty}_{1-i\infty}\frac{\zeta_F'}{\zeta_F}(s,\chi)\hat{g}(s)ds=\sum_{v}\frac{1}{4\pi i}\int^{1+\epsilon+i\infty}_{1+\epsilon-i\infty}\frac{L'_v}{L_v}(s,\chi)\hat{g}(s)ds+\frac{\delta_\chi}{4}g(1)
\]
for some $\epsilon>0$, and $\delta_\chi$ is 1 if $\chi$ is the trivial character, contributing half the residue at Re$(s)=1$, and zero otherwise. As in \cite[p.12]{W72}, the residue at 1 can be expressed as an integral over the global Weil group:
\[
\frac14\int_{W_F}g(|w|)\chi(w)\frac{dw}{|w|}.
\]
To each local $L$-factor we apply Lemmas \ref{weil1} and \ref{weil2} as follows: for $\sigma=\frac12$, we have
\begin{align*}
\frac{1}{4\pi i}\int_{-i\infty}^{i\infty} \hat{g}(s)\frac{d}{ds}\log L_v(1+s,\chi)
=-\frac12\log q_v\sum_{n=1}^\infty g(q_v^{-n})q_v^{-n}\int_{W^0_v}	\chi_v(\f^n	w)dw.
\end{align*}
On the other hand,
\[
\frac{1}{4\pi i}\int_{-\infty}^\infty \hat{g}(s)\frac{d}{ds}\log L_v(1-s,\overline{\chi})=-\frac12\log q_v\sum_{n=1}^\infty g(q_v^{n})q_v^{-n}\int_{W^0_v}\overline{\chi}_v(\f^n w)dw,
\]
plus the same contribution associated to the residue at 1, since $\overline{\chi}$ is trivial if $\chi$ is. Then using the fact that $\overline{\chi}(w)=\chi(w^{-1})$ and
\[
W_v-W_v^0=\bigcup_{n\in\Z-\{0\}}f^n_vW^0_v,
\]
we combine the two terms two obtain
\begin{align*}
&-\frac12\int_{W_v-W_v^0}\hat{g}(|w|)\chi_v(w)\inf(|w|,|w|^{-1})dw=-\frac12pv_0\int_{W_v-W_v^0}\hat{g}(|w|)\chi_v(w)\frac{|w|dw}{|1-w|}
\end{align*}
as in \cite[p.158]{Mor}. Then we apply Corollary \ref{nonarch} to combine the contribution of the nonarchimedean $L$-factors and the conductor,
\[
-\frac12pv_0\int_{W_v}\hat{g}(|w|)\chi_v(w)\frac{|w|}{|1-w|}dw,
\]
which is identical to the contribution of the archimedean factors.

Finally, putting this all together we have 
\[
g(0)\log|d_F|+\frac12\int_{W_F}g(|w|)\chi(w)\frac{dw}{|w|}-\frac12 pv_0\int_{W_v}\hat{g}(|w|)\chi_v(w)\frac{|w|}{|1-w|}dw
\]
as desired.
\end{proof}

\section{Application to lower bounds for sums over zeroes}

\label{sec5}

The presence of the sums over zeroes, or equivalently, the Weil distributions gives us an approach to the zeroes of $L$-functions through the continuous spectrum of $SL(2)$. In this section we give an example in the simplest case to illustrate how the distributions arising from Eisenstein series can be used to bound the sum over zeroes.

\subsection{Truncation}
The key to this analysis will be the following positvity result, proved by Arthur following an idea of Selberg, valid for reductive groups $G$. Note that the case we are interested is also covered by Remark \ref{pos}. First, we review the truncation operator, referring to the exposition of \cite[\S13]{A} and onwards for details. We will work over $\Q$, though the discussion also holds for number fields.

For simplicity, take $G=SL(2)$, and let $\phi$ be a locally bounded, measurable function on $G_\Q\backslash G_\A^1$, and $T$ a suitably regular point in the positive root space $\mathfrak a_0^+$ generated by the roots of the maximal torus $A$ of $G$. Define the truncation operator
\[
\Lambda^T\phi=\sum_P(-1)^{\text{dim}A_P/A_G}\sum_{\delta\in P_\Q\backslash G_\Q}\int_{N_P(\Q)\backslash N_P(\A)}\phi(n\delta x)\hat{\tau}_P(\log(H_P(\delta x)-T))dn
\]
where the outer sum is taken over parabolic subgroups $P$ of $G$, $\hat{\tau}_P$ is the characteristic function of the positive Weyl chamber associated to $P$, and $H_P(x)$ is the usual height function
\[
H_P(n\begin{pmatrix}a&0\\0&b\end{pmatrix}k)=\max(|a|,|b|).
\]
The inner sum is finite, while the integrand is a bounded function of $n$. In particular, if $\phi$ is in $L^2_\text{cusp}(G_\Q\backslash G_\A^1)$ then $\Lambda^T\phi=\phi$; also $\Lambda^T E(g,\phi,s)$ is square integrable. It is self-adjoint and idempotent, hence an orthogonal projection. 

Using this, Arthur shows that the spectral side
\[
\tr(\rho(f))=\sum_\chi J_\chi^T(f)=\sum_\chi\int_{G_\Q\backslash G_\A^1}\Lambda^T_2K_\chi(x,x)dx
\]
converges absolutely, where the subscript on $\Lambda_2^T$ indicates truncation with respect to the second variable, and the index $\chi$ corresponds to the $\eta$ in the case of $SL_2$ above. This gives what Arthur calls the coarse spectral expansion of the trace formula. We have the expression for the distribution $J_\chi^T(f)$ as
\begin{align*}
&\int_{G_\Q\backslash G_\A^1}\Lambda^T_2K_\chi(x,x)dx\\
&=\sum_P\frac{1}{n_P}\int_{G_\Q\backslash G_\A^1}\int_{i\mathfrak a_P^*}\sum_{\phi}\Lambda^T E(x,\rho(f,\eta,s)\phi,\lambda)\overline{\Lambda^TE(x,\phi,\lambda)}d\lambda\ dx
\end{align*}
where $n_P$ is the number of chambers in $\mathfrak a_P$, and $\phi$ runs over an orthonormal basis of $\rho(f,\eta,s)$.

In particular, when $G=SL(2)$, there is only one chamber, and the inner integral is one dimensional, that is, $i\mathfrak a_P^*=i\R$. Furthermore, by absolute convergence we may interchange the sum over $\phi$ with the integral, to obtain the expression for $SL_2$
\be
\label{truncated}
J_\chi^T(f)=\sum_\phi\int^\infty_{-\infty}(\Lambda^T E(x,\rho(f,\eta,s)\phi,s),\Lambda^TE(x,\phi,s))d|s|
\ee
thus we have an absolutely convergent sum-integrals of an inner product.

Now, the following is a consequence of Arthur's method:

\begin{lemma}
\label{arthurpos}
$J^T_\chi(f)$ is a positive-definite distribution.
\end{lemma}

\begin{proof}
Using Arthur's truncation $\Lambda^T$ with respect to the parameter $T$ and Arthur's general notation, the intertwining operator is given as
\[
(M^T_{P,\chi}(\lambda)\phi',\phi)=\int_{G_\Q\backslash G_\A^1}\Lambda^TE(x,\phi',\lambda)\overline{\Lambda^TE(x,\phi,\lambda)}dx
\]
for any vectors $\phi',\phi$ in the induced representation space. It is an integral of the usual $L^2$-inner product of truncated Eisenstein series (it is square integrable after truncation), so it follows that $M^T_{P,\chi}(\lambda)$ is positive-definite, self-adjoint operator.

Then, following the proof of the conditional convergence of $J_\chi^T$ in \cite[\S7]{A}, we see that for any positive-definite test function $f*f^*$ with $C_c^\infty(G_\A)$, the resulting double integral is nonnegative, and the integrals can be expressed as an increasing limit of nonnegative functions. The integral converges, and $J_\chi^T(f)$ is positive-definite.

Alternatively, using the absolute convergence of the spectral side, we may interchange the sum and integrals to obtain the inner product expression as above, and observe again that the inner product is positive-definite.
\end{proof}

\begin{remark}
We can also give a direct proof of positivity in the case where $F=\Q$ and with ramification only at infinity, which we outline here.\footnote{We are grateful to the referee for suggesting this.}  The Maa\ss-Selberg relation for a truncated Eisenstein series for $SL_2(\Z)$ is given by
\begin{align*}
&\Big\langle \Lambda^TE(\cdot,s_1) , \Lambda^TE(\cdot,s_2) \Big\rangle \\
&= \frac{T^{s_1+\bar{s}_2-1}}{s_1+\bar{s}_2-1} + \overline{m(s_2)} \frac{T^{s_1-\bar{s}_2}}{s_1-\bar s_2} + m(s_1) \frac{T^{\bar{s}_2-s_1}}{\bar s_2- s_1} + m(s_1)\overline{m(s_2)} \frac{T^{1-s_1-\bar{s}_2}}{1-s_1-\bar s_2}.
\end{align*}
If we set $s_1 = s_2 = \frac12(1+\epsilon + ir)$ with $r\in \R$ and $\epsilon>0$, then as $\epsilon$ tends to 0 the previous expression tends to 
\begin{align*}
&\lim_{\epsilon\to0}\Big\langle \Lambda^TE(\cdot,\frac{1+\epsilon + ir}{2}) , \Lambda^TE(\cdot,\frac{1+\epsilon + ir}{2}) \Big\rangle \\
&=2\log T + m\Big(\frac{1-ir}{2}\Big)\frac{T^{ir}}{ir} - m\Big(\frac{1+ir}{2}\Big)\frac{T^{-ir}}{ir} - \frac12\left(\frac{m'}{m}\Big(\frac{1-ir}{2}\Big) + \frac{m'}{m}\Big(\frac{1+ir}{2}\Big)\right).
\end{align*}
Then choosing a test function $g = g_0*g_0$ where $g_0$ belongs to $C_c^\infty(\R)$, so that
\[
\hat{g}(s) = \hat{g}_0(s)\hat{\bar{g}}_0(1-s).
\]
It follows then that 
\begin{align*}
&\int_{-\infty}^\infty \Big\langle \Lambda^TE(\cdot,1/2+ir) , \Lambda^TE(\cdot,1/2+ir) \Big\rangle\hat{g}(1/2 + ir) dr \label{epos}\\
&=\int_{-\infty}^\infty \Big\langle \Lambda^TE(\cdot,1/2+ir) , \Lambda^TE(\cdot,1/2+ir) \Big\rangle|\hat{g}_0(1/2 + ir)|^2 dr \notag
\end{align*}
since $\hat{\bar{g}}_0(s) = \overline{\hat{g}_0(\bar s)}$. Hence we conclude that 
\[
\int_{-\infty}^\infty \Big\langle \Lambda^TE(\cdot,1/2+ir) , \Lambda^TE(\cdot,1/2+ir) \Big\rangle\hat{g}(1/2 + ir) dr \ge0.
\]
\end{remark}

\subsection{Application to lower bounds}

Now the positivity of $J_\chi(f)$ gives immediately a lower bound for the sums over zeroes for any class $\chi$ and function $f$. We illustrate our method in the most basic case, that is, for the Riemann zeta function.

\begin{proof}
[Proof of Corollary \ref{mainthm4}]
The first inequality (\ref{lowerbdhecke}) immediately follows from Lemma \ref{arthurpos} for a fixed $\eta$, since the second term in (\ref{LL2}) is nonzero only if $\chi=\omega$. For the second (\ref{lowerbdzeta}), we consider the continuous spectral terms in the trace formula (\ref{truncated}), truncated at $T$. Specialising in the case of $SL(2,\R)$ to give, by the usual Maa\ss-Selberg relations
\[
\int_{\Gamma\backslash G}\Lambda^TK_\text{cont}(x,x)dx= g(1)\log T-\frac{1}{4\pi}\int^\infty_{-\infty}\frac{m'}{m}(it)\hat{g}(it)dt \label{P(T)}
+\frac{1}{4\pi i}\int^\infty_{-\infty}m(it)\hat{g}(it)\frac{T^{it}}{t}dt.
\]
From Lemma \ref{arthurpos}, for our choice of $g$ the above expression is nonnegative. The first term described in Theorem \ref{mainthm1},
\begin{align*}
&\sum_{\rho}\hat{g}(\rho)-\int_{0}^\infty \Big\{g(x)-\frac{1}{4}g^*(x)\Big\}dx +\sum_{n=1}^\infty \Lambda(n)g(n)\\
&+\int_1^\infty \Big\{g(x)+g^*(x)-\frac{2}{x}g(1)\Big\}\frac{x\ dx}{2(x^2-1)}+\frac12(\log 4\pi +\gamma)g(1).
\end{align*}
The requirement that $T> \sqrt{3}/{2}$ simply follows from the observation that the fundamental domain of $SL_2(\Z)\backslash {\bf H}_2$ has height at least $\sqrt{3}/{2}$ in ${\bf H}_2$. Then the claim follows immediately from rearranging the terms and evaluating at $T=1$. 
\end{proof}

Certainly we may evaluate at the usual point $T=1$, but it may be that for different test functions there will be more effective choices of $T$, which may be of interest to applications of the explicit formula. Finally, we leave it to the interested reader to examine the analogous result for Hecke $L$-functions.

\bibliographystyle{alpha}

\bibliography{master}

\end{document}